\documentclass[11pt,a4paper]{article}
\usepackage{mathrsfs}
\usepackage{multirow}
\usepackage{tikz}
 \usepackage{epsfig}
 \usepackage{epstopdf}
\usepackage[T1]{fontenc}
\usepackage{geometry}
\usepackage{amsmath}
\usepackage{bm}
\usepackage{amsbsy,latexsym,amsfonts, epsfig, color, authblk, amssymb, graphics, bm}
\usepackage{epsf,slidesec,epic,eepic}
\usepackage{fancybox}
\usepackage{fancyhdr}
\usepackage{setspace}
\usepackage{cases}
\usepackage{nccmath}
\usepackage{url}
\usepackage[colorlinks, citecolor=blue]{hyperref}

\newtheorem{theorem}{Theorem}[section]
\newtheorem{lemma}{Lemma}[section]

\newtheorem{definition}{Definition}[section]
\newtheorem{example}{Example}[section]
\newtheorem{proposition}{Proposition}[section]
\newtheorem{corollary}{Corollary}[section]
\newtheorem{assumption}{Assumption}[section]
\newtheorem{remark}{Remark}[section]

\newtheorem{alemma}{Lemma}

\newenvironment{proof}{{\noindent \bf Proof:}}{\hfill$\Box$\medskip}
\newenvironment{aproof}{{\noindent \bf Proof:}}{\hfill$\Box$\medskip}

\definecolor{lred}{rgb}{1,0.8,0.8}
\definecolor{lblue}{rgb}{0.8,0.8,1}
\definecolor{dred}{rgb}{0.6,0,0}
\definecolor{dblue}{rgb}{0,0,0.5}
\definecolor{dgreen}{rgb}{0,0.5,0.5}

\title{Subregularity of subdifferential mappings relative to the critical set and KL property of exponent 1/2}
  \author{Shaohua Pan\footnote{School of Mathematics, South China University of Technology, Guangzhou. shhpan@scut.edu.cn}
  \ \ {\rm and}\ \
 Yulan Liu\footnote{Corresponding author(ylliu@gdut.edu.cn), School of Applied Mathematics, Guangdong University of Technology, Guangzhou.}}

\begin{document}

 \maketitle
\begin{abstract}
  For a proper extended real-valued function, this work focuses on the relationship
  between the subregularity of its subdifferential mapping relative to the critical set
  and its KL property of exponent 1/2. When the function is lsc convex, we establish
  the equivalence between them under the continuous assumption on the critical set.
  Then, for the uniformly prox-regular function, under its continuity on
  the local minimum set, the KL property of exponent 1/2 on the local minimum set
  is shown to be equivalent to the subregularity of its subdifferential relative to this set.
  Moreover, for this class of nonconvex functions, under a separation assumption
  of stationary values, we show that the subregularity of its subdifferential
  relative to the critical set also implies its KL property of exponent $1/2$.
  These results provide a bridge for the two kinds of regularity,
  and their application is illustrated by examples.
 \end{abstract}
 \section{Introduction}

  For a proper extended real-valued function, the (metric) subregularity of its
  (limiting) subdifferential at a (limiting) critical point for the origin
  states that for any point near the critical point, its distance to
  the critical set is upper bounded by the remoteness of the subdifferential
  at this point. When the function is lower semicontinuous (lsc) convex,
  Artacho and Geoffroy \cite{Artacho08} provided an equivalent characterization
  for this property in terms of the quadratic growth of the function,
  and this property plays a crucial role in the linear convergence analysis of
  the first-order methods (see, e.g., \cite{Zhou17,Cui18}). However,
  when the function is nonconvex, there are few works on the subregularity
  of the subdifferential at critical points especially the relationships with
  other error bounds, except that Drusvyatskiy et al. \cite{Drusvyatskiy15}
  proved that the subregularity of the subdifferential at a local minimum
  implies the quadratic growth at this point (see also \cite[Remark 2.2(iii)]{Artacho14}).

  \medskip

  Recently, the KL property of a proper function is successfully applied for analyzing
  the global convergence of algorithms for nonconvex and nonsmooth optimization problems
  (see \cite{Attouch10,Attouch13,Bolte14}). In particular, the KL property of exponent 1/2
  is the key to achieve the linear convergence of the corresponding algorithms.
  From \cite[Corollary 2]{Ngai09}, we know that for a proper function continuous
  on the global minimum set, its KL property of exponent $1/2$ over the global
  minimum set implies its quadratic growth over this set. Then, it is natural to
  ask what is the relationship between the subregularity of the subdifferential
  at a critical point and the KL property of exponent 1/2 at the reference point.
  This work focuses on this so as to build a bridge for the two kinds of regularity.

  \medskip

  When the function is lsc convex, Theorem 5 of \cite{Bolte17} implies the equivalence
  between the KL property of exponent $1/2$ and an error bound that is weaker than
  the subregularity of the subdifferential. In Section \ref{sec3.1}, under the continuity
  on the critical set, we establish the equivalence between its KL property of exponent 1/2
  and the subregularity of its subdifferential relative to the critical set. Then,
  in Section \ref{sec3.2} we extend this equivalence to a class of nonconvex functions,
  i.e., for the uniformly prox-regular function, under its continuity on the local
  minimum set, the KL property of exponent 1/2 on the local minimum set is shown
  to be equivalent to the subregularity of its subdifferential relative to this set.
  In particular, for the uniformly prox-regular function, we also show that
  under a separation assumption of stationary values, the subregularity of its
  subdifferential relative to the critical set implies its KL property of exponent $1/2$.

  \medskip

  Uniformly prox-regular functions are abundant, including
  the locally Lipschitz prox-regular, primal-lower-nice (pln) and
  lower $C^2$-functions, the semiconvex function, and the sum of
  a proper closed convex function and a proper function that
  is smooth on its open domain (see Theorem \ref{theo-composite}).
  For the last structured nonconvex functions, Li and Pong \cite{LiPong18}
  showed that under \cite[Assumption B]{Luo93}, the Luo-Tseng error bound \cite{Tseng09}
  implies the KL property of exponent 1/2. Whereas our Theorem \ref{theo-composite}
  shows that under a little weaker assumption on stationary values,
  the subregularity of its subdifferential relative to the critical set
  implies the KL property of exponent 1/2. We also study the link between the subregularity
  of the subdifferential relative to the critical set and the Luo-Tseng error bound,
  and extend the result of \cite[Section 3]{Zhou17} partially to the nonconvex setting.
 \section{Notations and preliminaries}\label{sec2}

  We denote $\mathbb{X}$ by a finite-dimensional vector
  space equipped with the inner product $\langle \cdot,\cdot\rangle$ and its induced
  norm $\|\cdot\|$. For an extended real-valued function
  $f\!:\mathbb{X}\to\overline{\mathbb{R}}:=[-\infty,+\infty]$,
  $f$ is called proper if $f(x)>-\infty$ for each $x\in\mathbb{X}$
  and ${\rm dom}\,f:=\{x\in\mathbb{X}\,|\, f(x)<\infty\}\ne\emptyset$,
  and for given real numbers $\alpha$ and $\beta$, set
  $[\alpha\le f\le\beta]:=\{x\in\mathbb{X}\!: \alpha\le f(x)\le \beta\}$.
  For a proper $f\!:\mathbb{X}\to\overline{\mathbb{R}}$,
  we use $x'\xrightarrow[f]{}x$ to signify $x'\to x$ and $f(x')\to f(x)$.
  For a given $\overline{x}\in\mathbb{X}$ and a constant $\varepsilon>0$,
  $\mathbb{B}(\overline{x},\varepsilon)$ denotes the closed ball centered at $\overline{x}$
  of radius $\varepsilon$. For a proper convex function $\phi\!:\mathbb{X}\to\overline{\mathbb{R}}$,
  denote by ${\rm prox}_{\phi}$ its proximal operator, and for a closed set $S\subseteq\mathbb{X}$,
  the notation $\Pi_S$ denotes the projection operator on $S$.
 \subsection{Generalized subdifferentials and subregularity}\label{sec2.1}
 \begin{definition}\label{rsubgrad}(\cite[Definition 8.3 $\&$ 8.45]{RW98})
  Consider a function $f\!:\mathbb{X}\to\overline{\mathbb{R}}$ and a point
  $x$ with $f(x)$ finite. The regular subdifferential of $f$ at $x$ is defined by
   \[
    \widehat{\partial}\!f(x):=\bigg\{v\in\mathbb{X}\ |\
    \liminf_{x'\to x\atop x'\ne x}
    \frac{f(x')-f(x)-\langle v,x'-x\rangle}{\|x'-x\|}\ge 0\bigg\},
  \]
  and the (limiting) subdifferential $\partial\!f(x)$, horizon subdifferential
  $\partial^{\infty}f(x)$ and proximal subdifferential $\partial^P\!f(x)$
  of $f$ at $x$ are respectively defined by
  \begin{subequations}
  \[
    \partial\!f(x):=\Big\{v\in\mathbb{X}\ |\ \exists\,x^k\xrightarrow[f]{}x,
    v^k\to v\ {\rm with}\ v^k\in\widehat{\partial}\!f(x^k)\ {\rm for\ each}\ k\Big\},
  \]
  \[
    \partial^{\infty}f(x):=\Big\{v\in\mathbb{R}^p\ |\  \exists\,x^k\xrightarrow[f]{}x,
    \lambda^k\downarrow 0\ {\rm and}\ v^k\in\widehat{\partial}\!f(x^k)\ {\rm with}\ \lambda^kv^k\to v\ {\rm as}\ k\to\infty\Big\}.
  \]
  \[
   \partial^P\!f(x):=\bigg\{v\in\mathbb{X}\ |\
    \liminf_{x'\to x,x'\ne x}\frac{f(x')-f(x)-\langle v,x'-x\rangle}{\|x'-x\|^2}>-\infty\bigg\}.
  \]
 \end{subequations}
 \end{definition}
 \begin{remark}\label{remark-Fsubdiff}
  At each $x\in{\rm dom}f$, $\widehat{\partial}\!f(x)\subseteq\partial\!f(x)$,
  where the former is closed convex but the latter is generally nonconvex.
  When $f$ is convex, both $\widehat{\partial}\!f(x)$ and $\partial\!f(x)$
  reduce to the subdifferential of $f$ at $x$ in the sense of convex analysis.
  The point $\overline{x}$ at which $0\in\partial\!f(\overline{x})$
  is called a (limiting) critical point of $f$, and we denote by
  ${\rm crit}f$ the set of critical points.
 \end{remark}

 \begin{definition}\label{subregularity}
   Let $\mathcal{F}\!:\mathbb{X}\rightrightarrows\mathbb{X}$ be a multifunction.
   Consider an arbitrary $(\overline{x},\overline{y})\in{\rm gph}\mathcal{F}$.
   We say that $\mathcal{F}$ is (metrically) subregular at $\overline{x}$
   for $\overline{y}$ if there exist $\varepsilon>0$ and $\kappa>0$ such that
   for all $x\in\mathbb{B}(\overline{x},\varepsilon)$,
   \(
    {\rm dist}(x,\mathcal{F}^{-1}(\overline{y}))\le \kappa{\rm dist}(\overline{y},\mathcal{F}(x)).
   \)
  \end{definition}

  Definition \ref{subregularity} is a little different from the original one,
  and we here adopt an equivalent form by \cite[Section 3H]{DR09}.
  It is well-known that the subregularity of $\mathcal{F}$ at $\overline{x}$
  for $\overline{y}\in\mathcal{F}(\overline{x})$ iff its inverse $\mathcal{F}^{-1}$
  is calm at $\overline{y}$ for $\overline{x}\in\mathcal{F}^{-1}(\overline{y})$.
  For the recent discussions on the criterion of calmness and subregularity,
  the reader may refer to \cite{Gfrerer13,Henrion02}. In this work, we focus on
  the subregularity of the subdifferential of a proper function
  $f\!:\mathbb{X}\to\overline{\mathbb{R}}$, and say that $\partial\!f$ is
  subregular relative to ${\rm crit}f$ for the origin if it is subregular
  at each $x\in{\rm crit}f$ for the origin.
  Generally, it is not an easy task to check whether the subdifferential
  of a proper (even convex) function is subregular or not.
  Lemma \ref{KL-convex} of Appendix summarizes some (convex) functions
  whose subdifferentials are subregular at each point of their graphs.
 \subsection{Kurdyka-{\L}ojasiewicz property}\label{sec2.2}

 \begin{definition}\label{KL-Def1}
  Let $f\!:\mathbb{X}\!\to\overline{\mathbb{R}}$ be a proper function.
  The function $f$ is said to have the Kurdyka-{\L}ojasiewicz (KL) property
  at $\overline{x}\in{\rm dom}\,\partial\!f$ if there exist $\eta\in(0,+\infty]$,
  a continuous concave function $\varphi\!:[0,\eta)\to\mathbb{R}_{+}$ satisfying
  the following properties
  \begin{itemize}
    \item [(i)] $\varphi(0)=0$ and $\varphi$ is continuously differentiable on $(0,\eta)$;

    \item[(ii)] for all $s\in(0,\eta)$, $\varphi'(s)>0$,
  \end{itemize}
  and a neighborhood $\mathcal{U}$ of $\overline{x}$ such that
  for all
  \(
    x\in\mathcal{U}\cap\big[f(\overline{x})<f<f(\overline{x})+\eta\big],
  \)
  \[
  \varphi'(f(x)-f(\overline{x})){\rm dist}(0,\partial\!f(x))\ge 1.
  \]
  If the corresponding $\varphi$ can be chosen as $\varphi(s)=c\sqrt{s}$
  for some $c>0$, then $f$ is said to have the KL property at $\overline{x}$
  with an exponent of $1/2$. If $f$ has the KL property of exponent $1/2$
  at each point of ${\rm dom}\,\partial\!f$,
  then $f$ is called a KL function of exponent $1/2$.
 \end{definition}
 \begin{remark}\label{KL-remark}
  To argue that a proper $f$ is a KL function of exponent $1/2$,
  it suffices to check whether it has the KL property of $1/2$
  at all critical points or not, since by \cite[Lemma 2.1]{Attouch10} it has
  the property at any noncritical point.
 \end{remark}
 \subsection{Pln and uniformly prox-regular functions}\label{sec2.3}
 \begin{definition}\label{Def-pln}(see \cite[Definition 3.1]{Poliquin91} or \cite[Definition 2.1]{Levy95})
  Let $f\!:\mathbb{X}\to\overline{\mathbb{R}}$ be a proper function.
  We say that $f$ is primal-lower-nice (pln) at $\overline{x}\in{\rm dom}f$ if there exist
  $\overline{\rho}>0,\overline{c}>0$ and $\overline{\varepsilon}>0$ such that
  for all $x,y\in\mathbb{B}(\overline{x},\overline{\varepsilon})$,
  $\rho>\overline{\rho}$, $v\in\partial^P\!f(x)$ with $\|v\|\le\overline{c}\rho$,
  \[
   f(y)\ge f(x)+\langle v,y-x\rangle -\frac{\rho}{2}\|y-x\|^2.
  \]
  If $f$ is pln at each $x\in{\rm dom}f$, then we say that $f$ is a pln function.
 \end{definition}

 By \cite[Remark 1.5]{Marcellin06}, for a pln function $f$,
 $\partial^P\!f(x)=\widehat{\partial}\!f(x)=\partial\!f(x)=\partial^C\!f(x)$
 for all $x\in{\rm dom}f$, where $\partial^C\!f(x)$ is the Clarke subdifferential
 of $f$. That is, the definition of the pln property is independent of the involved
 subdifferential. The pln function includes strongly amenable functions
  and semiconvex functions. For strongly amenable functions, the reader is
  invited to see \cite[Section 10.F]{RW98}. A proper function
  $f\!:\mathbb{X}\to\overline{\mathbb{R}}$ is called semiconvex if there exists
  $\gamma\ge 0$ such that $x\mapsto f(x)+\frac{1}{2}\gamma\|x\|^2$ is convex.
 \begin{definition}\label{Uprox-regular}
  Let $f\!:\mathbb{X}\to\overline{\mathbb{R}}$ be a proper function.
  We say that $f$ is uniformly prox-regular at $\overline{x}\in{\rm dom}f$
  if there exist $\delta>0$ and $\rho\ge 0$ such that
  \begin{equation}\label{ineq-prox}
    f(y)\ge f(x)+\langle v, y-x\rangle - \frac{\rho}{2}\|y-x\|^2
  \end{equation}
  whenever $x,y\in\mathbb{B}(\overline{x},\delta)$ and $v\in\partial\!f(x)$
  with $f(x)\le f(\overline{x})+\delta$.
 \end{definition}

 The uniform prox-regularity of $f$ at $\overline{x}$ in Definition \ref{Uprox-regular}
 was introduced by Daniilidis et al. \cite{Daniilidis08}, which is weaker than
 the uniform prox-regularity of $f$ on some neighborhood of $\overline{x}$
 introduced in \cite{Bernard05}. By combining this observation with \cite[Proposition 3.9]{Bernard05},
 if $f$ is locally Lipschitz continuous at $\overline{x}$, the pln of $f$
 at $\overline{x}$ implies its uniform prox-regularity at $\overline{x}$.
 In addition, it is easy to check that if $f$ is continuous at $\overline{x}$,
 its uniform prox-regularity at $\overline{x}$ implies its pln at $\overline{x}$.
 This means that when $f$ is locally Lipschitz continuous at $\overline{x}$,
 the pln of $f$ at $\overline{x}$, the uniform prox-regularity of $f$ at $\overline{x}$,
 and the uniform prox-regularity of $f$ on some neighborhood of $\overline{x}$ are the same,
 which by \cite[Proposition 3.11]{Bernard05} are equivalent to the prox-regularity
 of $f$ at $\overline{x}$ as well as the lower-$C^2$ property of $f$ at $\overline{x}$.
 \section{Subregularity and KL property of exponent $1/2$}\label{sec3}

 In this section, for a proper $f\!:\mathbb{X}\to\overline{\mathbb{R}}$,
 we shall investigate the link between the subregularity of $\partial\!f$
 relative to ${\rm crit}f$ and the KL property of exponent ${1}/{2}$ of $f$.
 \subsection{The case that $f$ is convex}\label{sec3.1}

  For a proper lsc convex $f$ and a point $\overline{x}\in{\rm crit}f$,
  Theorem 5 of \cite{Bolte17} implies the equivalence
  between the KL property of exponent $1/2$ of $f$ at $\overline{x}$
  and the following error bound
  \begin{align}\label{errbound}
   \exists\,\varepsilon>0,r_0>0,c_0>0\ \ {\rm s.t.}\ \ {\rm for\ all}\ x\in[f(\overline{x})<f<r_0]\cap\mathbb{B}(\overline{x},\varepsilon),\nonumber\\
   f(x)-f(\overline{x})\ge c_0{\rm dist}^2(x,(\partial\!f)^{-1}(0)).\qquad\qquad\qquad
  \end{align}
  In this part, under the continuity of $f$ at $\overline{x}$,
  we achieve the equivalence among the subregularity of $\partial\!f$
  at $\overline{x}$ for the origin, the KL property of exponent $1/2$
  of $f$ at $\overline{x}$, and the following quadratic growth
  \begin{equation}\label{quadratic-growth}
   \exists\varepsilon>0,\nu>0\ \ {\rm s.t.}\ \
   f(x)-f(\overline{x})\ge\nu {\rm dist}^2(x,(\partial\!f)^{-1}(0))
  \ \ {\rm for\ all}\ x\in\mathbb{B}(\overline{x},\varepsilon).
  \end{equation}
  This requires the following lemma (see \cite{Bolte17}), which summarizes some favorable
  properties of the trajectories of the differential inclusion for a proper lsc
  convex function. Its proof was given in \cite{Brezis73} except for part (v),
  and the proof of part (v) was provided in \cite{Bruck75}.
 \begin{lemma}\label{convex-curve}
  Let $h\!:\mathbb{X}\to\overline{\mathbb{R}}$ be a proper lsc convex function with
  ${\rm crit}h\ne\emptyset$. For each $x\in{\rm dom}h$, there is a unique absolutely
  continuous curve $\chi_x\!:[0,+\infty)\to\!\mathbb{X}$ such that
  \begin{align}\label{inclusion}
  \left\{\begin{array}{l}
   \dot{\chi}_{x}(t)\in-\partial h(\chi_{x}(t))\ \ {\rm a.e.}\ {\rm on}\ (0,+\infty),\\
   \chi_{x}(0)=x.
   \end{array}\right.
  \end{align}
  Also, the curve $\chi_x$ (called a subgradient curve) has the following properties:
  \begin{itemize}
   \item[(i)] For all $t>0$, the right derivative $\frac{d}{dt}\chi_{x}(t^{+})$ of $\chi_x$
              is well defined and satisfies
              \[
                \frac{d}{dt}\chi_{x}(t^{+})=-\partial^{0}h(\chi_{x}(t)):=-\mathop{\arg\min}_{z\in\partial h(\chi_{x}(t))}\|z\|;
               \]
   \item[(ii)] $\frac{d}{dt}h(\chi_{x}(t^{+}))=-\|\dot{\chi}_{x}(t^{+})\|^2$ for all $t>0$;

   \item[(iii)] for each $z\in{\rm crit}h$, the function $t\mapsto \|\chi_{x}(t)-z\|$ decreases;

   \item[(iv)] the function $t\mapsto h(\chi_{x}(t))$ is nonincreasing and $\lim_{t\to\infty}h(\chi_{x}(t))=h^*$;

   \item[(v)] $\chi_{x}(t)$ converges to some $\widehat{x}\in{\rm crit}h$ as $t\to\infty$.
  \end{itemize}
 \end{lemma}
 \begin{theorem}\label{subregular1-KL}
  Let $f\!:\mathbb{X}\to\overline{\mathbb{R}}$ be a proper lsc convex function.
  Consider an arbitrary $\overline{x}\in{\rm crit}f$. Then, for the following
  three statements:
  \begin{itemize}
   \item[(i)] the multifunction $\partial\!f$ is subregular at $\overline{x}$ for the origin;

   \item[(ii)] $f$ satisfies the quadratic growth at $\overline{x}$ as in \eqref{quadratic-growth};

   \item[(iii)] $f$ has the KL property of exponent $1/2$ at $\overline{x}$;
  \end{itemize}
  it holds that $(i)\Leftrightarrow(ii)\Rightarrow(iii)$. If, in addition,
  $f$ is continuous at $\overline{x}$, then $(ii)\Leftrightarrow(iii)$.
 \end{theorem}
 \begin{proof}
  The equivalence $(i)\Leftrightarrow(ii)$ follows by \cite[Theorem 3.3]{Artacho08}.
  We next establish the implication $(ii)\Rightarrow(iii)$. Its proof is similar to
  that of \cite[Theorem 5(ii)]{Bolte17}, and we include it for completeness.
  By the quadratic growth $f$ at $\overline{x}$ as in \eqref{quadratic-growth},
  there exist $\varepsilon>0$ and $\nu>0$ such that
  \eqref{quadratic-growth} holds for all $x\in\mathbb{B}(\overline{x},\varepsilon)$.
  Pick up an arbitrary $x\in\mathbb{B}(\overline{x},\varepsilon)$.
  If $\partial\!f(x)=\emptyset$, the conclusion automatically holds.
  So, it suffices to consider the case $\partial\!f(x)\ne\emptyset$.
  From the convexity of $f$, for any $z\in{\rm dom}\,f$ and $\xi\in\partial\!f(x)$,
  $f(z)\ge f(x)+\langle\xi,z-x\rangle$. Then,
  \[
    f(x)-f(z)\le \inf_{\xi\in\partial f(x)}\|\xi\|\inf_{z\in{\rm crit}f}\|z-x\|
    \le {\rm dist}(x,(\partial f)^{-1}(0))\inf_{\xi\in\partial f(x)}\|\xi\|.
  \]
  By setting $z=\overline{x}$ in the last inequality and using the inequality \eqref{quadratic-growth},
  we obtain
  \[
    [f(x)-f(\overline{x})]^{1/2}
    \le \frac{1}{\sqrt{\nu}}\inf_{\xi\in\partial f(x)}\|\xi\|
    =\frac{1}{\sqrt{\nu}}{\rm dist}(0,\partial f(x)).
  \]
  This shows that $f$ has the KL property of exponent $1/2$ at $\overline{x}$.

  \medskip

  Now assume that $f$ is continuous on ${\rm crit}f$. We argue that
  $(iii)\Rightarrow(ii)$ holds. Since $f$ has the KL property of exponent $1/2$
  at $\overline{x}$ and $\overline{x}\in{\rm dom}\partial\!f$, there exist
  $\varepsilon_1>0,\eta>0$ and $c>0$ such that
  for all $x\in\mathbb{B}(\overline{x},\varepsilon_1)
  \cap[f(\overline{x})<f<f(\overline{x})+\eta]$,
  \[
    {\rm dist}(0,\partial\!f(x))\ge c\sqrt{f(x)-f(\overline{x})}.
  \]
  Since $f$ is continuous at $\overline{x}$,
  there exists $\varepsilon_2>0$ such that for all $x\in\mathbb{B}(\overline{x},\varepsilon_2)$,
  \[
    f(\overline{x})\le f(x)<f(\overline{x})+\eta.
  \]
  Set $\epsilon=\min(\varepsilon_1,\varepsilon_2)$. The last two inequalities imply that
  for all $x\in\mathbb{B}(\overline{x},\epsilon)\backslash{\rm crit}f$,
  \begin{equation}\label{KL-fineq}
    {\rm dist}(0,\partial\!f(x))\ge c\sqrt{f(x)-f(\overline{x})}.
  \end{equation}
  Since $f$ is a proper lsc and convex function, by Lemma \ref{convex-curve}
  for every $x\in{\rm dom}f$ there is a unique absolutely continuous
  subgradient curve $\chi_{x}\!:[0,+\infty)\to\mathbb{X}$ that satisfies
  the differential inclusion \eqref{inclusion}.
  Define $\omega(t):=\sqrt{f(\chi_{x}(t))-f^*}$ for $t\in[0,+\infty)$.
  Fix an arbitrary $x\in\mathbb{B}(\overline{x},\epsilon)\backslash{\rm crit}f$
  and consider the differential inclusion \eqref{inclusion}.
  By Lemma \ref{convex-curve}(iii) it follows that
  $\|\chi_{x}(t)-\overline{x}\|\le\|x-\overline{x}\|\le\epsilon$
  for any $t>0$. Fix an arbitrary $T>0$. Then,
  \[
    \frac{d\omega(t)}{dt}=\frac{\frac{d}{dt}f(\chi_{x}(t)}{2\sqrt{f(\chi_{x}(t))-f(\overline{x})}}
    =-\frac{\|\dot{\chi}_{x}(t)\|^2}{2\sqrt{f(\chi_{x}(t))-f(\overline{x})}}
    \le-\frac{c}{2}\|\dot{\chi}_{x}(t)\|\quad\forall t\in(0,T),
  \]
  where the second equality is due to Lemma \ref{convex-curve}(ii),
  and the inequality is using \eqref{KL-fineq} and the fact that
  $\dot{\chi}_{x}(t)\in-\partial f(\chi_{x}(t))$. The last inequality implies that
  \[
    \omega(T)-\omega(0)=\lim_{\epsilon\to 0^{+}}\int_{\epsilon}^{T}\frac{d}{dt}\omega(t)dt
    \le -\frac{c}{2}\int_{0}^{T}\|\dot{\chi}_{x}(t)\|dt
    \le -\frac{c}{2}{\rm dist}(\chi_{x}(T),\chi_{x}(0))
  \]
  where the last inequality is since the length of the curve connecting
  any two points is at least as long as the Euclidean distance between them.
  Thus, together with the nonnegativity of $\omega(T)$, it is immediate to obtain
  \[
    \sqrt{f(x)-f(\overline{x})}=\omega(0)\ge\frac{c}{2}{\rm dist}(\chi_{x}(T),\chi_{x}(0)).
  \]
  Notice that the function $\chi_{x}(\cdot)$ is continuous by \cite[Theorem 14]{Bolte10}.
  By taking the limiting $T\to\infty$ and using Lemma \ref{convex-curve}(v),
  it follows that
  \[
    \sqrt{f(x)-f(\overline{x})}=\omega(0)\ge \frac{c}{2}{\rm dist}(x,{\rm crit}f)=\frac{c}{2}{\rm dist}(x,(\partial f)^{-1}(0)).
  \]
  This, by the arbitrariness of $x\in\mathbb{B}(\overline{x},\epsilon)\backslash{\rm crit}f$,
  implies that part (ii) holds.
 \end{proof}
 \begin{remark}\label{remark-convex}
  By \cite[Theorem 3.1]{Dong09}, the assertion (i) is also equivalent to the subregularity
  of the proximal mapping $\mathcal{R}_f(x)\!:=x-{\rm prox}_f(x)$ at $\overline{x}$ for the origin.
  Thus, one may add the subregularity of $\mathcal{R}_f$ to the list of assertions.
  By Theorem \ref{subregular1-KL}, all convex functions in Lemma \ref{KL-convex}
  of Appendix are the KL function of exponent $1/2$.
 \end{remark}
 \subsection{The case that $f$ is nonconvex}\label{sec3.2}

  Denote by $\mathcal{S}_{f}^*$ the local minimum set of a proper function $f$.
  The following proposition states the relation among the KL property of
  exponent $1/2$ on $\mathcal{S}_{f}^*$, the quadratic growth of $f$
  as in \eqref{quadratic-growth} over $\mathcal{S}_{f}^*$,
  and the subregularity of $\partial\!f$ relative to this set for the origin.
 \begin{proposition}\label{KL-subregular}
  Let $f\!:\mathbb{X}\to\overline{\mathbb{R}}$ be a proper function.
  Consider an arbitrary $\overline{x}\in\mathcal{S}_{f}^*$ at which
  $f$ is continuous. If $f$ has the KL property of exponent $1/2$ at $\overline{x}$,
  then it satisfies the quadratic growth as in \eqref{quadratic-growth}
  at $\overline{x}$ and $\partial\!f$ is subregular at $\overline{x}$ for the origin.
 \end{proposition}
 \begin{proof}
  In order to show that $f$ satisfies the quadratic growth as in \eqref{quadratic-growth}
  at $\overline{x}$, we first argue that the condition of
  \cite[Corollary 2(ii)]{Ngai09} holds at $\overline{x}$.
  Since $f$ has the KL property of exponent $1/2$ at $\overline{x}$,
  there exist $\varepsilon'>0, \eta>0$ and $c>0$ such that
  \begin{equation}\label{temp-KL12}
    {\rm dist}(0,\partial\!f(z))\ge c\sqrt{f(z)-f(\overline{x})}
  \end{equation}
  for all $z\in\mathbb{B}(\overline{x},\varepsilon')\cap[f(\overline{x})<f<f(\overline{x})+\eta]$.
  Since $f$ is continuous at $\overline{x}$ and $\overline{x}$ is a local minimizer of $f$,
  there exists $\varepsilon''>0$ such that for all $z\in\mathbb{B}(\overline{x},\varepsilon'')$,
  \begin{equation}\label{equa-continuous}
    f(\overline{x})\le f(z)\le f(\overline{x})+\eta/2.
  \end{equation}
  Define $S_g:=\big\{z\in\mathbb{X}\ |\ g(z):=f(z)+\delta_{\mathbb{B}(\overline{x},\varepsilon'')}(z)-f(\overline{x})\le 0\big\}$.
  Set $\epsilon=\min(\varepsilon',\varepsilon'')/2$.
  Fix an arbitrary $x\in\mathbb{B}(\overline{x},\epsilon)\backslash S_g$.
  From \eqref{temp-KL12}, we have ${\rm dist}(0,\partial\!f(x))\ge c\sqrt{f(x)-f(\overline{x})}$.
  Take arbitrary $v^1\in\partial^{\infty}\!f(x)$ and $v^2\in\partial^{\infty}\delta_{\mathbb{B}(\overline{x},\varepsilon'')}(x)$
  such that $v^1+v^2=0$. From \cite[Proposition 8.12]{RW98}, it follows that
  \(
    \partial^{\infty}\delta_{\mathbb{B}(\overline{x},\varepsilon'')}(x)
    =\mathcal{N}_{\mathbb{B}(\overline{x},\varepsilon'')}(x)=\{0\}
  \)
  where the second equality is due to $x\in{\rm int}\,\mathbb{B}(\overline{x},\varepsilon'')$.
  This means that $v^2=v^1=0$. By \cite[Corollary 10.9]{RW98},
  \[
    \widehat{\partial}g(x)\subseteq\partial g(x)
    \subseteq \partial\!f(x)+\mathcal{N}_{\mathbb{B}(\overline{x},\varepsilon'')}(x)
    =\partial\!f(x).
  \]
  This, together with  ${\rm dist}(0,\partial\!f(x))\ge c\sqrt{f(x)-f(\overline{x})}$,
  immediately implies that
  \[
   {\rm dist}(0,\widehat{\partial}g(x))\ge c\sqrt{f(x)-f(\overline{x})}=c\sqrt{g(x)}.
  \]
  By the arbitrariness of $x$ in $\mathbb{B}(\overline{x},\epsilon)\backslash S_g$,
  the condition of \cite[Corollary 2(ii)]{Ngai09} is satisfied
  with $\gamma=1/2$, $a=c/2$ and $\varepsilon=\epsilon$. Applying the conclusion of
  \cite[Corollary 2(ii)]{Ngai09} yields
  \[
    \max(g(x),0)\ge a^2{\rm dist}^2(x,S_g)\quad\ {\rm for\ all}\
    x\in\mathbb{B}(\overline{x},\epsilon).
  \]
 Notice that each point of $S_g$ is a local minimizer of $f$. Clearly,
 $S_g\subset{\rm crit}f=(\partial\!f)^{-1}(0)$. Then, the last inequality
 shows that $f$ has the quadratic growth as in \eqref{quadratic-growth}
 at $\overline{x}$, i.e.,
 \begin{equation}\label{equa1-QG}
   f(x)\ge f(\overline{x})+\frac{c^2}{4}{\rm dist}^2(x,(\partial\!f)^{-1}(0))
   \quad\ \forall x\in\mathbb{B}(\overline{x},\epsilon/2)
 \end{equation}

 Next we show that $\partial\!f$ is subregular at $\overline{x}$ for the origin.
 Fix an arbitrary $x\in\mathbb{B}(\overline{x},\epsilon/2)$.
 If $\partial\!f(x)=\emptyset$, then the following inequality holds for any $\kappa>0$:
 \begin{equation}\label{aim-ineq}
    {\rm dist}(x, (\partial\!f)^{-1}(0))\le\kappa{\rm dist}(0,\partial\!f(x)).
 \end{equation}
 Now assume that $\partial\!f(x)\ne\emptyset$.
 If $x\in[f(\overline{x})<f<f(\overline{x})+\eta]$,
 from \eqref{temp-KL12} and \eqref{equa1-QG} we have
 \[
   {\rm dist}(0,\partial\!f(x))\ge \frac{c^2}{2}{\rm dist}(x,(\partial\!f)^{-1}(0)),
 \]
 i.e., the inequality \eqref{aim-ineq} holds with $\kappa=\sqrt{2}/c$.
 When $x\notin[f(\overline{x})<f<f(\overline{x})+\eta]$, the inequality \eqref{equa-continuous}
 implies that $f(x)=f(\overline{x})$. Together with \eqref{equa1-QG},
 ${\rm dist}^2(x,(\partial\!f)^{-1}(0))=0$, and the inequality \eqref{aim-ineq} holds
 for any $\kappa>0$. To sum up, the inequality \eqref{aim-ineq} holds
 for any $x\in\mathbb{B}(\overline{x},\epsilon/2)$. Consequently, $\partial\!f$
 is metrically subregular at $\overline{x}$ for the origin.
 \end{proof}
 \begin{remark}
  The continuity of $f$ at $\overline{x}$ is necessary for the conclusion
  of Proposition \ref{KL-subregular}. For example, consider
  \(
    f(x):=\left\{\begin{array}{cl}
      x^4 &{\rm if}\ x\neq 0;\\
      -1 & {\rm if}\ x=0
   \end{array}\right.
  \)
  for $x\in\mathbb{R}$. Clearly, $\overline{x}=0$ is a local minimizer of $f$
  but $f$ is not continuous at $\overline{x}$. A simple calculation yields
  $\partial\!f(0)\!=\mathbb{R}$ and $\partial\!f(x)\!=\!\{4x^3\}$ if $x\ne 0$.
  One may check that $\partial\!f$ is not subregular at $\overline{x}$
  though $f$ satisfies the quadratic growth condition as in \eqref{quadratic-growth}
  and has the KL property of exponent $1/2$ at $\overline{x}$.
 \end{remark}

 Observe that the quadratic growth of a proper $f$ at $\overline{x}$ as in \eqref{quadratic-growth}
 implies that $\overline{x}$ is a local minimum. Then, from Proposition \ref{KL-subregular},
 we obtain the following conclusion.
 \begin{corollary}\label{corollary-subregular}
  Let $f\!:\mathbb{X}\to\overline{\mathbb{R}}$ be a proper function.
  Consider an arbitrary $\overline{x}\in{\rm crit}f$ at which
  $f$ is continuous and satisfies the quadratic growth as in \eqref{quadratic-growth}.
  If $f$ has the KL property of exponent $1/2$ at $\overline{x}$,
  then $\partial\!f$ is subregular at $\overline{x}$ for the origin.
 \end{corollary}
 \begin{remark}\label{remark1-nonconvex}
  As \cite[Remark 2.2(iii)]{Artacho14} pointed out, for a proper lsc
  $f\!:\mathbb{X}\to\overline{\mathbb{R}}$ and an arbitrary local minimum $x^*$
  of $f$, the subregularity of $\partial\!f$ at $x^*$ for the origin implies
  the quadratic growth of $f$ at $x^*$ as in \eqref{quadratic-growth}.
  This result was also proved in \cite[Theorem 4.3]{Drusvyatskiy15}.
  In fact, for those $\overline{x}\in{\rm crit}f$ at which there exists
  $\delta>0$ such that $f(y)=f(\overline{x})$ for all
  $y\in\mathbb{B}(\overline{x},\delta)\cap{\rm crit}f$,
  by following the same arguments as in \cite[Theorem 3.3]{Artacho08},
  the subregularity of $\partial\!f$ at $\overline{x}$ for the origin
  also implies its quadratic growth at $\overline{x}$ as in \eqref{quadratic-growth}.
  So, the quadratic growth of $f$ at $\overline{x}$ is
  necessary for the result of Corollary \ref{corollary-subregular}.
  \end{remark}

 For a general $f$, it is not clear whether the converse of Corollary
 \ref{corollary-subregular} holds or not, but when $f$ is uniformly prox-regular
 at $\overline{x}$, we can achieve it under a suitable assumption.
 \begin{proposition}\label{KL-Uprox}
  Let $f\!:\mathbb{X}\to\overline{\mathbb{R}}$ be a proper function. Consider
  an arbitrary $\overline{x}\in{\rm crit}f$ for which there exists $\delta>0$
  such that $f(y)\le f(\overline{x})$ for all $y\in\mathbb{B}(\overline{x},\delta)\cap{\rm crit}f$.
  If $f$ is uniformly prox-regular at $\overline{x}$ and $\partial\!f$
  is subregular at $\overline{x}$ for the origin,
  then $f$ has the KL property of exponent ${1}/{2}$ at $\overline{x}$.
 \end{proposition}
 \begin{proof}
  Since the function $f$ is uniformly prox-regular at $\overline{x}$, there exist $\delta'>0$
  and $\rho>0$ such that for all $x,y\in\mathbb{B}(\overline{x},\delta')$
  and all $v\in\partial\!f(x)$ with $f(x)\le f(\overline{x})+\delta'$,
  \begin{equation}\label{Uprox}
    f(y)\ge f(x)+\langle v,y-x\rangle-\frac{\rho}{2}\|y-x\|^2.
  \end{equation}
  Since $\partial\!f$ is subregular at $\overline{x}$ for the origin,
  there exist $\varepsilon'>0$ and $\kappa>0$ such that the inequality \eqref{aim-ineq}
  holds for all $x\in\mathbb{B}(\overline{x},\varepsilon')$.
  Set $\epsilon=\min(\delta,\varepsilon',\delta')/2$. Consider an arbitrary $z\in\mathbb{B}(\overline{x},\epsilon)\cap[f(\overline{x})<f<f(\overline{x})+\epsilon]$.
  When $\partial\!f(z)=\emptyset$, for any $c>0$, we have
  \begin{equation}\label{tempKL-ineq}
    {\rm dist}(0,\partial\!f(z))\ge c\sqrt{f(z)-f(\overline{x})}.
  \end{equation}
  Now assume that $\partial\!f(z)\ne\emptyset$. Take an arbitrary $\xi\in\partial\!f(z)$
  and let $u\in\Pi_{{\rm crit}f}(z)$. Then, $\|u-\overline{x}\|\le
  \|u-z\|+\|z-\overline{x}\|\le 2\|z-\overline{x}\|\le\delta'$.
  By invoking \eqref{Uprox} with $y=u$ and $x=z$,
  \begin{equation*}
    f(u)\ge f(z)+\langle\xi,u-z\rangle -\frac{\rho}{2}\|u-z\|^2.
  \end{equation*}
  Notice that $f(u)\le f(\overline{x})$ since $\|u-\overline{x}\|\le2\|z-\overline{x}\|\le\delta$.
  Together with the last inequality,
  \begin{align*}
    f(z)-f(\overline{x})&\le f(z)-f(u)
    \le\inf_{\xi\in\partial\!f(z)}\langle\xi,z-u\rangle+\frac{\rho}{2}\|z-u\|^2\\
    &\le\inf_{\xi\in\partial\!f(z)}\|\xi\|\|z-u\|+\frac{\rho}{2}\|z-u\|^2\\
    &={\rm dist}(0,\partial\!f(x)){\rm dist}(z,(\partial\!f)^{-1}(0))+\frac{\rho}{2}{\rm dist}^2(z,(\partial\!f)^{-1}(0)).
  \end{align*}
  Since ${\rm dist}(z,(\partial\!f)^{-1}(0))\le\kappa{\rm dist}(0,\partial\!f(z))$
  by \eqref{aim-ineq} with $x=z$, the last inequality implies
  \[
   f(z)-f(\overline{x})\le (\kappa+{\rho\kappa^2}/{2}){\rm dist}^2(0,\partial\!f(z)).
  \]
  Along with $f(z)>f(\overline{x})$, we have that $f$ has the KL property of exponent $1/2$ at $\overline{x}$.
 \end{proof}

 From the discussions after Definition \ref{Uprox-regular}, when $f$ is
 locally Lipschitz continuous at $\overline{x}$, the uniform prox-regularity
 of $f$ at $\overline{x}$ in Proposition \ref{KL-Uprox} can be replaced by
 the prox-regularity of $f$ at $\overline{x}$, and consequently,
 by the pln of $f$ at $\overline{x}$. The following assumption is about
 the separation of stationary values, which is a little weaker than
 \cite[Assumption 4.1]{LiPong18} or \cite[Assumption B]{Luo93}. Clearly,
 it holds at a local maximizer of $f$. By \cite[Lemma 3.1]{Luo92},
 it also holds for the quadratic functions on the polyhedral set.
 \begin{assumption}\label{assump1}
  For each $\overline{x}\in{\rm crit}f$, there exists $\delta>0$ such that
  $f(y)\le f(\overline{x})$ for all $y\in\mathbb{B}(\overline{x},\delta)\cap{\rm crit}f$.
 \end{assumption}

 From Proposition \ref{KL-subregular} and \ref{KL-Uprox} and
 Corollary \ref{KL-subregular}, we obtain the following conclusion.
 \begin{theorem}\label{theorem-nonconvex}
  For a proper uniformly prox-regular $f\!:\mathbb{X}\to\overline{\mathbb{R}}$,
  the following results hold.
  \begin{itemize}
    \item[(i)] If $f$ is continuous on the local minimum set,
          its KL property of exponent $1/2$ on this set is equivalent to
          the subregularity of $\partial\!f$ relative to this set for the origin;

   \item[(ii)] Under Assumption \ref{assump1}, the subregularity of $\partial\!f$
               relative to ${\rm crit}f$ for the origin implies that $f$ is a KL function
               of exponent $1/2$. The converse also holds if $f$ is continuous on ${\rm crit}f$
               and satisfies the quadratic growth as in \eqref{quadratic-growth} on ${\rm crit}f$.
  \end{itemize}
 \end{theorem}

  When $f$ is a proper semiconvex function, it is easy to check that the subregularity
  of $\partial\!f$ relative to ${\rm crit}f$ is equivalent to the existence
  of a constant $\gamma\ge 0$ such that $\mathcal{R}_{f_{\gamma}}(x):=x-{\rm prox}_{f_{\gamma}}(x+\gamma x)$
  is subregular relative to ${\rm crit}f$, where $f_{\gamma}(\cdot)=f(\cdot)+\frac{\gamma}{2}\|\cdot\|^2$.
  In this case, the subregularity of $f$ relative to ${\rm crit}f$ can be replaced by that of
  $\mathcal{R}_{f_{\gamma}}(x)$.

  \medskip

  Figure 1 and 2 below summarize the relations among
  the KL property of exponent $1/2$, the quadratic growth, and the subregularity
  of subdifferential for a proper $f\!:\mathbb{X}\to\overline{\mathbb{R}}$ at a
  local minimum $x^*$ and a critical point $\overline{x}$, respectively.

   \noindent
 \vspace{-0.5cm}
 \begin{figure}[htb]
   \includegraphics[width=0.95\textwidth]{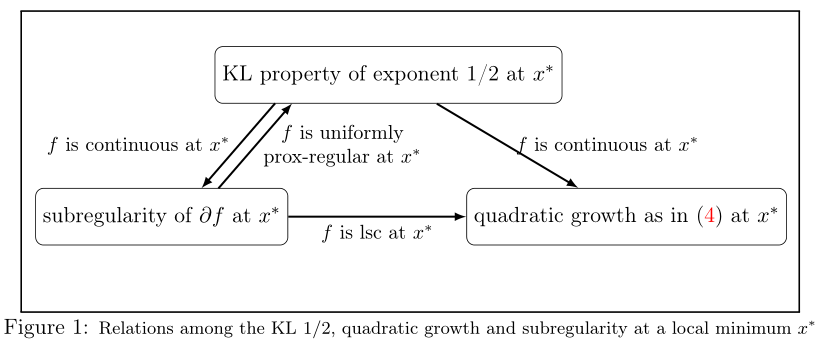}
 \end{figure}

\noindent
 \vspace{-0.5cm}
 \begin{figure}[htb]
   \includegraphics[width=0.95\textwidth]{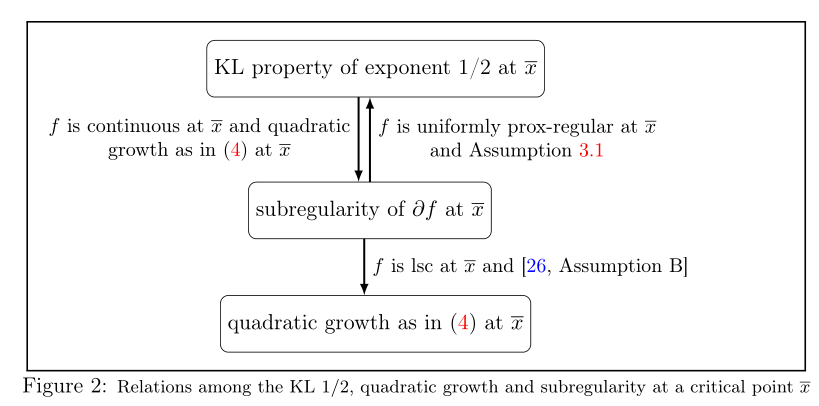}
 \end{figure}

 \subsection{The case that $f$ is structured nonconvex}\label{sec3.3}

  In this part we concentrate on the structured nonconvex $f$ that has
  the following form
  \begin{equation}\label{composite1}
   f(x):=g(x)+h(x),
  \end{equation}
  where $g\!:\mathbb{X}\to(-\infty,+\infty]$ is a proper function with
  an open domain and is smooth on ${\rm dom}g$,
  and $h\!:\mathbb{X}\to(-\infty,+\infty]$ is a proper function with
  ${\rm dom}h\cap{\rm dom}\,g\ne\emptyset$. For this class of structured
  nonconvex functions, we have the following conclusion.
 \begin{theorem}\label{theo-composite}
  For the function $f$ in \eqref{composite1} with a uniformly prox-regular $h$,
  consider an $\overline{x}\in{\rm crit}f$ for which there exists
  $\delta>0$ such that $f(y)\le f(\overline{x})$ for all
  $y\in\mathbb{B}(\overline{x},\delta)\cap{\rm crit}f$.
  Then the subregularity of $\partial\!f$ at $\overline{x}$ for $0$
  implies the KL property of exponent $1/2$ at $\overline{x}$.
 \end{theorem}
 \begin{proof}
  By Proposition \ref{KL-Uprox}, it suffices to argue that $f$ is uniformly
  prox-regular at $\overline{x}$. Since $\overline{x}\in{\rm dom}g$ and
  $g$ is smooth on ${\rm dom}g$, it follows that $\nabla g$ is strictly continuous
  at $\overline{x}$, i.e., there exist $\delta_1>0$ and $L_g>0$ such that
  $\mathbb{B}(\overline{x},\delta_1)\subset{\rm dom}g$ and
  \[
    \|\nabla g(y')-\nabla g(y)\|\le L_g\|y'-y\|\quad\ \forall y',y\in\mathbb{B}(\overline{x},\delta_1).
  \]
  Together with the mean-valued theorem, it follows that for all $y',y\in\mathbb{B}(\overline{x},\delta_1)$,
  \begin{align}\label{gineq}
    g(y')-g(y)&=\langle\nabla g(y),y'-y\rangle+\langle\nabla g(y+\overline{t}(y'-y))-\nabla g(y),y'-y\rangle\nonumber\\
    &\ge  \langle\nabla g(y),y'-y\rangle-L_g\|y'-y\|^2\ \ {\rm for\ some}\ \overline{t}\in(0,1).
  \end{align}
  In addition, since $h$ is uniformly prox-regular at $\overline{x}$,
  there exist $\delta_2>0$ and $\rho>0$ such that
  for all $z',z\in\mathbb{B}(\overline{x},\delta_2)$ and all $v\in\partial h(z)$
  with $h(z)\le h(\overline{x})+\delta_2$,
  \begin{equation}\label{hineq}
    h(z')-h(z)\ge\langle v,z'-z\rangle-\frac{\rho}{2}\|z'-z\|^2.
  \end{equation}
  Set $\overline{\delta}=\frac{1}{2}\min(\delta_1,\delta_2)$.
  Clearly, $\mathbb{B}(\overline{x},\overline{\delta})\subseteq{\rm dom}g$.
  Fix arbitrary $x',x\in\mathbb{B}(\overline{x},\overline{\delta})$ and an arbitrary
  $\xi\in\partial f(x)$ with $f(x)\le f(\overline{x})+\overline{\delta}$.
  From the continuity of $g$ at $\overline{x}$, we have $g(\overline{x})-g(x)<\delta_2/2$
  (if necessary by shrinking $\delta_1$). Together with $f(x)\le f(\overline{x})+\overline{\delta}$,
  we get
  \[
    h(x)\le h(\overline{x})+g(\overline{x})-g(x)\le h(\overline{x})+\delta_2.
  \]
  Since $\xi\in\partial f(x)=\nabla g(x)+\partial h(x)$, there exists $\zeta\in\partial h(x)$
  such that $\xi=\nabla g(x)+\zeta$. Now by invoking \eqref{hineq} with $z'=x'$ and $z=x$ and \eqref{gineq} with $y'=x'$ and $y=x$,
  it follows that
  \[
    f(x')\ge f(x)+\langle \nabla g(x)+\zeta,x'-x\rangle-\frac{2L_g\!+\!\rho}{2}\|x'-x\|^2.
  \]
  This shows that $f$ is uniformly prox-regular at $\overline{x}$.
  The proof is completed.
 \end{proof}

  Notice that the semiconvexity of $h$ implies its uniform prox-regularity.
  Hence, the conclusion of Theorem \ref{theo-composite} holds for the
  function $f$ in \eqref{composite1} with a semiconvex $h$. In this case,
  $f$ can be rewritten as $f=\widetilde{g}+\widetilde{h}$ where $\widetilde{g}$
  has the same property as $g$ does and $\widetilde{h}\!:\mathbb{X}\to(-\infty,+\infty]$
  is a proper convex function. Recently, for the nonconvex $f$ in \eqref{composite1}
  with a convex $h$, Li and Pong \cite{LiPong18} showed that under a proper
  separation assumption of stationary values (see \cite[Assumption B]{Luo93} or
  \cite[Assumption 4.1]{LiPong18}),
  the KL property of exponent $1/2$ is also implied by the Luo-Tseng
  error bound \cite{Tseng09} which is stated as: ``for any $\zeta\ge{\inf}_{z\in\mathbb{X}}f(z)$,
  there exist $\widetilde{\varepsilon}>0$ and $\widetilde{\kappa}>0$ such that
  \begin{equation}\label{Luo-error}
    {\rm dist}(x,(\partial\!f)^{-1}(0))\le \widetilde{\kappa}\|{\rm prox}_h(x-\nabla g(x))-x\|
  \end{equation}
  whenever $f(x)\le\zeta$ and $\|{\rm prox}_h(x-\nabla g(x))-x\|<\widetilde{\varepsilon}$''.
  It is natural to ask what is the link between the subregularity of $\partial\!f$
  relative to ${\rm crit}f$ and the Luo-Tseng error bound. When $g$ is convex,
  the results of \cite[Section 3]{Zhou17} actually show that under the compactness
  of ${\rm crit}f$, the former implies the latter. Here we extend this result to the nonconvex $g$.

  \begin{proposition}\label{error-subregular}
   For the function $f$ in \eqref{composite1} with a convex $h$, the following results hold.
   \begin{itemize}
    \item [(i)] Under Assumption \ref{assump1}, the Luo-Tseng error bound implies
                that $f$ is a KL function of exponent $1/2$.

    \item [(ii)] Under the continuity of $f$ on ${\rm crit}f$,
                 the Luo-Tseng error bound implies the subregularity of $\partial\!f$
                 relative to ${\rm crit}f$.

    \item[(iii)] If the critical set ${\rm crit}f$ is compact, the subregularity
                of $\partial\!f$ relative to ${\rm crit}f$ implies the following error bound:
                there exist $\tau>0$ and $\varpi>0$ such that
               \begin{equation}\label{aerror}
                {\rm dist}(z,(\partial\!f)^{-1}(0))\le\varpi\|{\rm prox}_h(z-\nabla g(z))-z\|
               \end{equation}
               for all $z\in{\rm dom}f$ with ${\rm dist}(z,(\partial\!f)^{-1}(0))\le\tau$.
               So, the Luo-Tseng error bound holds.
  \end{itemize}
  \end{proposition}
  \begin{proof}
   Write $\mathcal{R}(x):={\rm prox}_h(x-\nabla g(x))-x$ for $x\in\mathbb{X}$.
   Notice that $\partial\!f(x)=\nabla g(x)+\partial h(x)$ for each $x\in{\rm dom}g\cap{\rm dom}h$.
   Clearly, $x\in{\rm crit}f$ if and only if $\mathcal{R}(x)=0$.
   Part (i) follows by the proof of \cite[Theorem 4.1]{LiPong18}.
   We next prove that part (ii) holds. Fix an arbitrary $\overline{x}\in{\rm crit}f$.
   Since $\nabla\!g$ is strictly continuous at $\overline{x}$,
  there exist $\delta'>0$ and $L>0$ such that
  $\mathbb{B}(\overline{x},\delta')\subseteq{\rm dom}g$ and
  $\|\nabla g(z)-\nabla g(\overline{x})\|\le L\|z-\overline{x}\|$
  for all $z\in\mathbb{B}(\overline{x},\delta')$, and hence
  \(
   \|\mathcal{R}(x)\|=\|\mathcal{R}(x)-\mathcal{R}(\overline{x})\|\le (L+2)\|x-\overline{x}\|.
  \)
   Since the Luo-Tseng error bound holds,
   there exist $\widetilde{\varepsilon}>0$ and $\widetilde{\kappa}>0$ such that
   \(
   {\rm dist}(z,(\partial\!f)^{-1}(0))\le \widetilde{\kappa}\|\mathcal{R}(z)\|
   \)
   whenever $f(z)\le f(\overline{x})+1$ and $\|\mathcal{R}(z)\|<\widetilde{\varepsilon}$.
   Since $f$ is continuous at $\overline{x}$, there exists $\delta''>0$ such that
   $f(z)\le f(\overline{x})+1$ for all $z\in\mathbb{B}(\overline{x},\delta'')$.
   Set $\varepsilon=\min(\delta',\delta'',\frac{\widetilde{\varepsilon}}{L+3})$.
   For any $x\in\mathbb{B}(\overline{x},\varepsilon)$, since
   $f(x)\le f(\overline{x})+1$ and
   \(
   \|\mathcal{R}(x)\|<\widetilde{\varepsilon},
   \)
   it follows that
   \(
     {\rm dist}(x,(\partial\!f)^{-1}(0))\le \widetilde{\kappa}\|\mathcal{R}(x)\|,
   \)
   which by \cite[Lemma 4.1]{LiPong18} implies
   ${\rm dist}(x,(\partial\!f)^{-1}(0))\le \widetilde{\kappa}{\rm dist}(0,\partial f(x))$.
   So, $\partial\!f$ is subregular at $\overline{x}$ for the origin.

  \medskip
  \noindent
  {\bf(iii)} Since ${\rm crit}f\subseteq{\rm dom}g$, by the compactness of ${\rm crit}f$
  and the openness of ${\rm dom}g$, there exists $\delta'>0$ such that
  $\{u\in\mathbb{X}\!:{\rm dist}(u,{\rm crit}f)\le\delta'\}\subseteq {\rm dom}g$.
  By the smoothness of $g$ on ${\rm dom}g$, by shrinking $\delta>0$ if necessary,
  there exists $L>0$ such that $\|\nabla g(z)-\nabla g(z')\|\le L\|z-z'\|$
  for all $z,z'\in\{u\in\mathbb{X}\!:{\rm dist}(u,{\rm crit}f)\le\delta'\}$.
  In addition, since $\partial\!f$ is metrically subregular relative to
  the compact ${\rm crit}f$, by following the similar arguments as those
  for \cite[Proposition 2]{Zhou17}, there exist $\delta''>0$ and $\kappa''>0$ such that
  for all $x\in {\rm crit}f+\delta''\mathbb{B}_{\mathbb{X}}$,
  \begin{equation}\label{aineq1}
   {\rm dist}(x,(\partial\!f)^{-1}(0))\le\kappa''{\rm dist}(0,\partial\!f(x)).
   \end{equation}
  Fix an arbitrary $z\in{\rm dom}f$ with ${\rm dist}(z,(\partial\!f)^{-1}(0))
  \le\min(\frac{\delta''}{3+L},\delta')$.
  Notice that
  \begin{align*}
   {\rm dist}(z+\mathcal{R}(z),{\rm crit}f)
   &={\rm dist}(z+\mathcal{R}(z),(\partial\!f)^{-1}(0))
    \le{\rm dist}(z,(\partial\!f)^{-1}(0))+\|\mathcal{R}(z)\|\\
   &={\rm dist}(z,(\partial\!f)^{-1}(0))+\|\mathcal{R}(z)-\mathcal{R}(z^*)\|
   \ \ {\rm with}\ z^*\in\Pi_{{\rm crit}f}(z)\\
   &\le {\rm dist}(z,(\partial\!f)^{-1}(0)) + (L+2)\|z-z^*\|\\
   &\le (L+3){\rm dist}(z,(\partial\!f)^{-1}(0))\le\delta'',
   \end{align*}
   where the first inequality is using $\|\nabla g(z)-\nabla g(z^*)\|\le L\|z-z^*\|$.
   By \eqref{aineq1} we have
   \begin{align*}
     {\rm dist}(z+\mathcal{R}(z),(\partial\!f)^{-1}(0))
     &\le\kappa''{\rm dist}(0,\partial\!f(z+\mathcal{R}(z)))\\
     &=\kappa''\min_{\eta\in\partial h(z+\!\mathcal{R}(z))}\|\eta+\nabla g(z+\!\mathcal{R}(z))\|.
   \end{align*}
   Notice that
   \(
     -\mathcal{R}(z)-\nabla g(z)\in\partial h({\rm prox}_h(z-\nabla g(z)))
     =\partial h(z+\mathcal{R}(z)).
   \)
   Then,
   \begin{align*}
    {\rm dist}(z+\mathcal{R}(z),(\partial\!f)^{-1}(0))
    &\le\kappa''\|-\mathcal{R}(z)-\nabla g(z)+\nabla g(z+\mathcal{R}(z))\|\\
    &\le \kappa''(1+L)\|\mathcal{R}(z)\|,
   \end{align*}
  which implies that ${\rm dist}(z,(\partial\!f)^{-1}(0))\le[\kappa''(1+\!L)+1]\|\mathcal{R}(z)\|$.
  By the arbitrariness of $z$, \eqref{aerror} holds with $\tau=\min(\frac{\delta''}{3+L},\delta')$
  and $\kappa'=[\kappa''(1+\!L)+1]$. The last part follows by using
  the same arguments as those for \cite[Proposition 3]{Zhou17}.
  The proof is completed.
 \end{proof}
 \begin{remark}
  For the function $f$ in \eqref{composite1} with a convex $h$ and a $C^1$-smooth
  $g\!:\mathbb{R}^n\to\mathbb{R}$ and Lipschitz continuous gradient $\nabla\!g$,
  Section 8 of \cite{Drusvyatskiy18} establishes the equivalence between
  the subregularity of $\partial\!f$ at $\overline{x}\in{\rm crit}f$
  for the origin and the subregularity of $\mathcal{R}$ at $\overline{x}\in\mathcal{R}^{-1}(0)$
  for the origin. This means that the conclusions (i) and (ii) of Lemma \ref{error-subregular}
  still hold when the subregularity of $\partial\!f$ is replaced by that of the proximal
  mapping $\mathcal{R}$.
 \end{remark}
 \section{Application of main result}\label{sec4}

  We illustrate the applications of Theorem \ref{theorem-nonconvex}
  or \ref{theo-composite} in achieving the KL property of exponent $1/2$
  for the zero-norm regularized quadratic homogenous functions.
 \begin{example}\label{exam31}
  Let $C$ be the zero-norm constraint set in Lemma \ref{deltaC} of Appendix B and
  $M$ be a $p\times p$ matrix. The function
  $f(x)\!:=\frac{1}{2}x^{\mathbb{T}}Mx+\delta_{C}(x)$ has the KL property of
  exponent $1/2$ at each $\overline{x}\in{\rm crit}f$ with $\|\overline{x}\|_0=\kappa$.
  Indeed, by Lemma \ref{deltaC} of Appendix B, the function $f$ is uniformly prox-regular
  at $\overline{x}$. Since $\partial\!f(x)=\frac{1}{2}(M+M^{\mathbb{T}})x+\mathcal{N}_C(x)$
  for all $x\in C$, from the characterization of $\mathcal{N}_C(x)$ in \cite[Theorem 3.9]{Bauschke14},
  we know that $\partial\!f$ is polyhedral, which by \cite[Proposition 1]{Robinson81}
  shows that $\partial\!f$ is metrically subregular at each point of its graph.
  In addition, for each $z\in{\rm crit}f$,
  there exists $u\in\mathcal{N}_C(z)$ such that $\frac{1}{2}(M+\!M^{\mathbb{T}})z+u=0$,
  and hence $f(z)=-\frac{1}{2}\langle z,u\rangle=0$,
  where the last equality is by the expression of $\mathcal{N}_C(z)$.
  This implies that Assumption \ref{assump1} holds. Thus, by invoking
  Theorem \ref{KL-Uprox}, $f$ has the KL property of exponent $1/2$ at $\overline{x}$.
  It is worthwhile to point out that the KL property of exponent $1/2$
  for $f$ at $\overline{x}$ can not be identified by \cite[Corollary 5.2]{LiPong18}
  since $f$ is not continuous at $\overline{x}$.
  \end{example}
  \begin{example}\label{exam32}
  Let $\Gamma\!:=\{Z\in\mathbb{R}^{n\times m}\ |\ \|Z\|_0\le\kappa\}$
  for an integer $\kappa\ge 1$. Consider
  \[
    F(U,V):=\langle U,\mathcal{A}(V)\rangle +\delta_{\Gamma}(U)+\delta_{\Gamma}(V)\quad\ \forall U,V\in\mathbb{R}^{n\times m}
  \]
  where $\mathcal{A}(V)\!:=AV$ for some $A\in\mathbb{R}^{n\times n}$.
  By arranging the variables $U$ and $V$ as the vectors and using the arguments
  as for Example \ref{exam31}, we conclude that $F$ has the KL property of
  exponent $1/2$ at those $(\overline{U},\overline{V})\in{\rm crit}F$ with
  $\|\overline{U}\|_0=\|\overline{V}\|_0=\kappa$.
  \end{example}
  \begin{example}\label{exam33}
   Consider $f(x):=\frac{1}{2}x^{\mathbb{T}}Mx+\delta_{C\cap\mathbb{R}_{+}^p}(x)$
   where $M$ is same as Example \ref{exam31}. Using the similar arguments as for
   Example \ref{exam31} and  Lemma \ref{deltaCR+} of Appendix B, we conclude that $f$
   has the KL property of exponent $1/2$ at those $\overline{x}\in{\rm crit}f$
   with $\|\overline{x}\|_0=\kappa$.
 \end{example}
 \section{Conclusions}\label{sec4}

  For a proper extended real-valued function, we have disclosed the relation
  between the subregularity of its subdifferential relative to the critical set
  and its KL property of exponent $1/2$, and the link between the former
  and the Luo-Tseng error bound for a class of structured nonconvex functions.
  We illustrate by examples the application of the obtained results in seeking
  new KL functions of exponent $1/2$.

 \bigskip
  \noindent
  {\large\bf Acknowledgements} \ \  This work is supported by the National Natural Science
  Foundation of China under project No.11571120 and No.11701186. The first author would like to
  thank Professor Ting Kei Pong for pointing out a gap in the original manuscript.

 \bigskip
 \noindent
 {\bf Appendix}
 \begin{alemma}\label{KL-convex}
  The subdifferential mappings of the following functions are subregular
  at every point of their graphs:
  \begin{itemize}
    \item[(i)] the piecewise linear function;

    \item[(ii)] the piecewise linear-quadratic convex function and its conjugate;

    \item[(iii)]  the vector $\ell_p$-norm with $p\in[1,2]\cup\{+\infty\}$;

    \item[(iv)]  the indictor function on the $p$-order cone with $p\in[1,2]\cup\{+\infty\}$;

    \item[(v)] the indicator function on the $C^2$-cone reducible self-dual closed convex cone;

    \item[(vi)] the spectral function $g(\lambda(X))$ for $X\in\!\mathbb{S}^n$ where $g\!:\mathbb{R}^n\to(-\infty,+\infty]$
                is a proper lsc symmetric convex function and $\partial g$ is metrically subregular
                at any $(\overline{x},\overline{y})\in{\rm gph}\partial g$;

   \item[(vii)] the spectral function $g(\sigma(X))$ for $X\in\!\mathbb{R}^{n_1\times n_2}$
                where $g\!:\mathbb{R}^n\to(-\infty,+\infty]$
                is a proper lsc absolute symmetric convex function with
                $n=\min(n_1,n_2)$ and $\partial g$ is metrically subregular
                at any $(\overline{x},\overline{y})\in{\rm gph}\partial g$.
 \end{itemize}
 \end{alemma}
 \begin{aproof}
  Part (i) follows by \cite[Proposition 1]{Robinson81};
  part (ii) follows from \cite{Sun86} or \cite[Proposition 12.30]{RW98};
  part (iii) is due to \cite{Zhou15}; part (iv) can be found in \cite{SunBiPan18};
  part (v) is due to \cite[Theorem 2.1]{LiuPan18};
  and part (vi)-(vii) follows by \cite[Proposition 15]{Cui17}.
  \end{aproof}
 \begin{alemma}\label{deltaC}
  Let $C\!:=\{x\in\mathbb{R}^p\ |\ \|x\|_0\!\le\kappa\}$ for a positive
  integer $\kappa$. Then, the indicator function $\delta_C$ is pln and
  and uniformly prox-regular at those $x\in\mathbb{R}^p$ with $\|x\|_0=\kappa$.
 \end{alemma}
 \begin{proof}
  Fix an arbitrary $\overline{x}\in\mathbb{R}^p$ with $\|\overline{x}\|_0=\kappa$.
  Clearly, there exists $\epsilon>0$ such that for all $z\in\mathbb{B}(\overline{x},\epsilon)$,
  ${\rm supp}(z)\supseteq{\rm supp}(\overline{x})$. Choose $\overline{\rho}=1$
  and $c=1$. Fix an arbitrary $x\in\mathbb{B}(\overline{x},\epsilon)$.
  Take $\rho>\overline{\rho}$ and an arbitrary $v\in\mathcal{N}_C^{P}(x)$
  with $\|v\|\le c\rho$, where $\mathcal{N}_C^{P}(x)$ is the proximal
  normal cone to $C$ at $x$. Clearly, $x\in C$. Together with
  ${\rm supp}(x)\supseteq{\rm supp}(\overline{x})$ and $\|\overline{x}\|_0=\kappa$,
  it follows that ${\rm supp}(x)={\rm supp}(\overline{x})$.
  Fix an arbitrary $y\in\mathbb{B}(\overline{x},\epsilon)$.
  If $y\in C$, together with ${\rm supp}(y)\supseteq{\rm supp}(\overline{x})$
  we obtain ${\rm supp}(y)={\rm supp}(\overline{x})$. Notice that
  $\mathcal{N}_C^{P}(x)=\mathcal{N}_C(x)=[\![x]\!]^{\perp}$ (see \cite{Bauschke14}).
  Then, $\langle v,y-x\rangle=0$. Consequently, it holds that
  \begin{equation}\label{f-ineq}
   \delta_{C}(y)\ge \delta_{C}(x) + \langle v,y-x\rangle -\frac{\rho}{2}\|y-x\|^2;
  \end{equation}
  when $y\notin C$, this inequality holds since $\delta_C(y)=+\infty$.
  This shows that $\delta_C$ is pln at $\overline{x}$. The above arguments
  also show that, for any $x,y\in\mathbb{B}(\overline{x},\epsilon)$ and any
  $v\in\mathcal{N}_C(x)$ with $\delta_C(x)\le \delta_C(\overline{x})+\epsilon$,
  the inequality \eqref{f-ineq} holds. So, $f$ is uniformly prox-regular at $\overline{x}$.
 \end{proof}
 \begin{alemma}\label{deltaCR+}
  The indicator function $\delta_{C\cap\mathbb{R}_{+}^p}$ of $C\cap\mathbb{R}_{+}^p$
  is uniformly prox-regular at those $x\in C\cap\mathbb{R}_{+}^p$ with $\|x\|_0=\kappa$,
  where $C$ is same as Lemma \ref{deltaC}.
 \end{alemma}
 \begin{proof}
  Notice that the multifunction $\mathcal{F}(x):=(C-x)\times(\mathbb{R}_{+}^p-x)$
  for $x\in\mathbb{R}^p$ is polyhedral. From \cite[Section 3.1]{Ioffe08},
  it follows that $\mathcal{N}_{C\cap\mathbb{R}_{+}^p}(x)
  \subseteq\mathcal{N}_{C}(x)+\mathcal{N}_{\mathbb{R}_{+}^p}(x)$ for
  $x\in C\cap\mathbb{R}_{+}^p$. Consider an arbitrary $\overline{x}\in\!C\cap\mathbb{R}_{+}^n$
  with $\|\overline{x}\|_0=\kappa$. Clearly, there exists $\epsilon>0$ such that
  for all $z\in\mathbb{B}(\overline{x},\epsilon)$,
  ${\rm supp}(z)\supseteq{\rm supp}(\overline{x})$. Fix arbitrary
  $x,y\in\mathbb{B}(\overline{x},\epsilon)$ and an arbitrary
  $v\in\mathcal{N}_{C}(x)+\mathcal{N}_{\mathbb{R}_{+}^p}(x)$ with
  $\delta_{C\cap\mathbb{R}_{+}^p}(x)\le \delta_{C\cap\mathbb{R}_{+}^p}(\overline{x})+\epsilon$.
  Clearly, $x\in C\cap\mathbb{R}_{+}^p$. Together with ${\rm supp}(x)\supseteq{\rm supp}(\overline{x})$
  and $\|\overline{x}\|_0=\kappa$, we have  ${\rm supp}(x)={\rm supp}(\overline{x})=\kappa$.
  Hence, $\mathcal{N}_{C}(x)=[\![x]\!]^{\perp}$. Since
  $\mathcal{N}_{\mathbb{R}_{+}^p}(x)\subseteq[\![x]\!]^{\perp}$, it follows that
  $v\in\mathcal{N}_{C}(x)+\mathcal{N}_{\mathbb{R}_{+}^p}(x)\subseteq[\![x]\!]^{\perp}$.
  If $y\in C$, along with ${\rm supp}(y)\supseteq{\rm supp}(\overline{x})$,
  we have ${\rm supp}(y)={\rm supp}(x)$, and consequently
  \[
    \delta_{C\cap\mathbb{R}_{+}^p}(y)\ge \delta_{C\cap\mathbb{R}_{+}^p}(x) + \langle v,y-x\rangle -\frac{\rho}{2}\|y-x\|^2
  \]
  holds for any $\rho\ge 0$; when $y\notin C$, this inequality automatically holds.
  This shows that $f$ is uniformly prox-regular at $\overline{x}$.
  The result follows by the arbitrariness of $\overline{x}$.
 \end{proof}

\end{document}